\documentclass[12pt]{article}
\usepackage{titlesec}
\usepackage{authblk}

\usepackage{amssymb}
\usepackage{amsmath}
\usepackage{amsthm}
\usepackage{xr}
\usepackage{xcolor}
\usepackage{amsmath}
\usepackage{subcaption}
\usepackage{amsfonts,dsfont,mathrsfs}
\usepackage{booktabs,multirow,array}
\usepackage{bm}
\usepackage{placeins}
\usepackage{natbib}
\usepackage{todonotes}
\usepackage{latexsym}
\usepackage{comment}
\usepackage[inline]{enumitem}
\usepackage{hyperref}
\usepackage{graphicx}
\usepackage{placeins}
\usepackage{bbm}
\usepackage{bm}
\usepackage{cleveref}
\usepackage{subcaption}
\usepackage{soul}

\newtheorem{theorem}{Theorem}
\newtheorem{corollary}{Corollary}

\newtheorem{remark}{Remark}%
\newtheorem{definition}{Definition}

\newcommand{\e}{{\rm e}}

\newcommand{\deq}{\stackrel{\mbox{\scriptsize d}}{=}}
\newcommand{\indicator}{\ensuremath{\mathbbm{1}}}

\newcommand{\calC}{\mathcal{C}}
\newcommand{\calM}{\mathcal{M}}

\newcommand{\calX}{\mathcal{X}}

\newcommand{\dd}{\mathrm d}

\newcommand{\partder}[1]{\frac{\partial}{ \partial #1}}
\newcommand{\E}{\mathsf{E}}

\newcommand{\X}{\mathbb{X}}

\newcommand{\M}{\mathbb{M}}

\newcommand{\R}{\mathbb{R}}

\newcommand{\prob}{\mathsf{P}}

\providecommand{\keywords}[1]{
  \small 
  \textbf{\textsc{Keywords:}} #1
  \normalsize
}

\usepackage{authblk}
\date{}

\title{On the Palm distribution of superposition of point processes}

\author[1]{Mario Beraha}
\author[2]{Federico Camerlenghi}
\affil[1]{\normalsize{Department of Mathematics, Politecnico di Milano, Italy}} 
\affil[2]{\normalsize{Department of Economics, Management and Statistics, University of Milano-Bicocca, Italy}}

\begin{document}

\maketitle

\begin{abstract}
Palm distributions are critical in the study of point processes. In the present paper we focus on a point process $\Phi$ defined as the superposition, i.e., sum, of two independent point processes, say $\Phi = \Phi_1 + \Phi_2$, and we characterize its Palm distribution. In particular, we show that the Palm distribution of $\Phi$ admits a simple mixture representation depending only on the Palm distribution of $\Phi_j$, as $j=1, 2$, and the associated moment measures. Extensions to the superposition of multiple point processes, and higher order Palm distributions, are treated analogously.
\end{abstract}


\section{Introduction} \label{sec1}

Palm distributions are fundamental tools in the study of point processes, especially in the analysis of spatial patterns and stochastic modeling. These distributions describe the conditional behavior of a point process $\Phi$ given that one or more of its points, or \emph{atoms}, are located at specific positions. This concept is critical for both theoretical exploration and practical applications, such as spatial statistics and stochastic geometry \citep{MoWaBook03, coeurjolly2017tutorial, BaBlaKa}.

A common operation in the analysis of point processes is the \emph{superposition} of two independent processes $\Phi_1$ and $\Phi_2$, where the resulting process $\Phi = \Phi_1 + \Phi_2$ encompasses the atoms of both original processes. Superposition is a standard technique in point process theory and is widely used across various applied fields \citep{DaVeJo1, MoWaBook03}. Despite its importance, the behavior of Palm distributions under superposition has not been thoroughly explored, especially in terms of how the Palm distributions of the individual processes, i.e., $\Phi_1$ and $\Phi_2$, contribute to that of the final process process $\Phi$. The present paper aims to fill this gap by showing a simple and elegant result: the Palm distribution of the superposed process can be expressed as a mixture of the Palm distributions of the two independent underlying processes. This mixture representation involves the Palm distributions of $\Phi_1$ and $\Phi_2$, as well as their respective moment measures. Our findings extend naturally to the superposition of more than two point processes and to higher-order Palm distributions.\\

The paper is structured as follows. We introduce the background and notation in Section \ref{sec:notation} to state the main result of our paper concenring the superposition of two point processes (Section \ref{sec:main}). The proof of this theoretical finding is deferred to Section \ref{sec:proof}.

\section{Background and notation} \label{sec:notation}

Througout the paper,  we will adopt the notation of \cite{BaBlaKa}.
Let $(\X, d)$ be a complete and separable metric space and denote by $\calX = \mathcal B(\X)$ be the associated Borel $\sigma$-algebra. A counting measure $\mu$ is a locally-finite measure on $(\X, \calX)$ such that $\mu(B) \in \mathbb N$ for any relatively compact set $B \in \calX$.
Let $\M_\X$ be the set of counting measures on $(\X, \calX)$ and $\mathcal M_\X$ be the smallest $\sigma$-algebra which makes the mappings $\mu \mapsto \mu(B)$ measurable, for any $B \in \calX$. Finally, we denote by $(\Omega, \mathcal A, \prob)$ the underlying probability space.\\

First of all we remind the notion of a point process.
\begin{definition}   
    A point process $\Phi$ is a measurable map $\Phi: (\Omega,  \mathcal A) \rightarrow (\M_\X, \mathcal M_\X)$. Its probability distribution is given by $\prob_\Phi = \prob \circ \Phi^{-1}$.
\end{definition}
We remind that a  point process is said to be \emph{simple} if the following condition is satisfied:
\[
\prob(\forall x \in \X, \Phi(\{x\}) \leq 1) = 1. \]
It is also worth recalling that 
a point process $\Phi$ can be represented as a sum of Dirac delta masses $\Phi = \sum_{j \geq 1} \delta_{X_j}$, where the sequence $(X_{j})_{j \geq 1}$ is a \emph{measurable enumeration} of $\Phi$. In particular, we refer to   \cite[Chapter 9]{DaVeJo2} for a detailed discussion on enumerations.\\

We will now recall the definition of moment measures and factorial moment measures of a point process $\Phi$.
The (first) moment measure $M_{\Phi}$ of $\Phi$ is defined by $M_{\Phi}(B) = \E[\Phi(B)]$ for $B \in \calX$. 
Higher order extensions can be easily derived. To this end, given a point process $\Phi = \sum_{j \in \mathbb Z} \delta_{X_j}$, we denote its $k$-th power as 
\begin{equation}\label{eq:k-mom}
    \Phi^k = \sum_{(j_1, \ldots, j_k) \in \mathbb Z^k} \delta_{(X_{j_1}, \ldots, X_{j_k})}.
\end{equation}
The associated $k$-th order moment measure of $\Phi$, denoted as $M_{\Phi}^k$, is defined by
\[
    M_{\Phi^k}(B_1 \times \cdots \times B_k) = \E[\Phi^k(B_1 \times \cdots \times B_k)] = \E[\Phi(B_1) \cdots \Phi(B_k)], 
\]
for any $B_1, \ldots , B_k \in \calX$.
In addition, the $k$-th \emph{factorial} measure of $\Phi$ is defined similarly to \eqref{eq:k-mom}, but we sum only on $k$-tuples of pairwise distinct points. Formally, we define:
\[
    \Phi^{(k)} = \sum_{(j_1, \ldots, j_k) \in \mathbb Z^{(k)}} \delta_{(X_{j_1}, \ldots, X_{j_k})},
\]
where $\mathbb Z^{(k)} = \{(j_1, \ldots, j_k) \in \mathbb Z^k\colon j_\ell \neq j_m, \text{ for any } \ell \neq m\}$.
The $k$-th factorial moment measure of $\Phi$, denoted as $M_{\Phi^{(k)}}$, is defined by
\begin{equation*}
       M_{\Phi^{(k)}}(B_1 \times \cdots \times B_k) := \E[\Phi^{(k)}(B_1 \times \cdots \times B_k)] ,
\end{equation*}
for any $B_1, \ldots , B_k \in \calX$.\\

\emph{Palm distributions} are a basic tool used to describe the conditional law of $\Phi$ given the location of one of its atoms. See \cite{coeurjolly2017tutorial} for an historic account of contributions to Palm theory and for further intuition. Here, we follow the more abstract (but general) treatment of \cite{DaVeJo2} and \cite{BaBlaKa}.
Define the \emph{Campbell} measure of $\Phi$ by
\[
    \mathcal{C}_\Phi(B \times L) = \E[\Phi_B \indicator_L (\Phi)], \quad B \in \calX, L \in \calM_X .
\]
Observe that $M_{\Phi}(B) = \calC_{\Phi}(B \times \M_\X)$. Then by virtue of the measure disintegration theorem \citep[see, e.g.,][]{Kallenberg2021} we have the following definition.
\begin{definition}[Palm distribution]
    Consider the Campbell measure of $\Phi$. If $M_{\Phi}$ is $\sigma$-finite, then $\mathcal{C}_{\Phi}$ admits a disintegration
    \[
       \mathcal{C}_{\Phi}(B \times L) = \int_B \prob_\Phi^x(L) M_\Phi(\dd x).
    \]
    Then, $\prob_\Phi^x(\cdot)$ is called a Palm distribution of $\Phi$ at $x$ and $\{\prob_\Phi^x\}_{x \in \X}$ is called a family of Palm distribution of $\Phi$.
\end{definition}

Since for any $x \in \X$, $\prob_\Phi^x$ is a probability distribution over $(\M_\X, \calM_\X)$, it can be identified with the law of some point process, say $\Phi_x \sim \prob_\Phi^x$, called a \emph{Palm version} of $\Phi$ at $x$. 

\begin{remark}\label{rem:conditional_palm}
By Proposition 3.1.12 in \cite{BaBlaKa}, $\prob(\Phi_x(\{x\}) \geq 1) = 1$ for $M_\Phi$-almost all $x \in \X$. This justifies the interpretation of the Palm distribution at $x$ as the probability distribution of $\Phi$ conditionally to $\Phi$ having an atom at $x$.     
\end{remark}

\begin{remark}[Reduced Palm kernel]
    In light of Remark \ref{rem:conditional_palm}, the point process $\Phi^!_x := \Phi_x - \delta_x$ is well defined, and $\Phi^!_x$ is called the \emph{reduced Palm} version of $\Phi$ at $x$. 
\end{remark}

It is possible to extend the definition of Palm distribution to multiple \emph{conditioning} points $\bm x = (x_1, \ldots, x_k) \in \X^k$. The resulting Palm distribution is interpreted as the probability distribution conditionally to $\Phi$ having $k$ atoms at the locations $x_1, \ldots, x_k$.  Under some technical conditions on factorial moment measures of $\Phi$ (see Section 3.3.2 in \citep{BaBlaKa}), we have
\begin{equation}\label{eq:palm_algebra}
    \left(\Phi^!_{\bm x}\right)^!_{\bm y} \deq \Phi^!_{(\bm x, \bm y)}, \quad (\bm x, \bm y) \in \X^k \times \X^\ell.
\end{equation}

\section{Main result} \label{sec:main}

Here, we state the main theorem of the present paper, which  concerns the Palm distribution of the superposition of two point processes. 

\begin{theorem}[Main result]\label{teo:main}
    Let $\Phi_1$ and $\Phi_2$ be two independent point processes. Then, the  Palm kernel $(\Phi_1 + \Phi_2)_{x}$ can be expressed as a mixture, i.e.
    \[
        (\Phi_1 + \Phi_2)_{x} \deq \begin{cases}
            \Phi_{1 x} + \Phi_2 & \text{with probability } \frac{M_{\Phi_1}(\dd x)}{M_{\Phi_1}(\dd x) + M_{\Phi_2}(\dd x)} \\
            \Phi_{1} + \Phi_{2 x} & \text{with probability } \frac{M_{\Phi_2}(\dd x)}{M_{\Phi_1}(\dd x) + M_{\Phi_2}(\dd x)}
        \end{cases}
    \]
\end{theorem}

\begin{remark}[Reduced Palm kernel of superposition]
Exploiting the relation $\Phi_{j x} \deq \Phi_{j x}^! + \delta_x$, for $M_{\Phi_j}$-almost all $x$, $j=1, 2$, we obtain an equivalent statement of \Cref{teo:main} in terms of the reduced Palm kernel, namely:
    \[
        (\Phi_1 + \Phi_2)^!_{x} \deq \begin{cases}
            \Phi^!_{1 x} + \Phi_2 & \text{with probability } \frac{M_{\Phi_1}(\dd x)}{M_{\Phi_1}(\dd x) + M_{\Phi_2}(\dd x)} \\
            \Phi_{1} + \Phi^!_{2 x} & \text{with probability } \frac{M_{\Phi_2}(\dd x)}{M_{\Phi_1}(\dd x) + M_{\Phi_2}(\dd x)}
        \end{cases}
    \]
\end{remark}

Theorem \ref{teo:main} can be extended in several directions. In particular, we can describe the Palm distribution of the superposition of more than two point processes, as clarified in the following corollary.
\begin{corollary}
     Let $\Phi_1, \ldots, \Phi_k$ be independent point processes, then
    \[
        \left(\sum_{j=1}^k \Phi_j\right)_x \deq \Phi_{j x} + \sum_{\ell \neq j} \Phi_\ell \quad \text{with probability proportional to } M_{\Phi_j}(\dd x).
    \]
    Moreover, the same results hold true for the corresponding reduced Palm kernels.
\end{corollary}

The case of multiple conditioning points can be formally addressed by means of \eqref{eq:palm_algebra}. In particular we can state the following.
\begin{corollary}
Let $\Phi_1$ and $\Phi_2$ be two independent point processes.
    For any $(x, y) \in \X^2$, we have
    \[
    (\Phi_1 + \Phi_2)^!_{(x, y)} \deq \begin{cases}
        \Phi^!_{1(x, y)} + \Phi_2 &\text{with probability } \propto M_{\Phi_1^{(2)}}(\dd x \, \dd y)\\
        \Phi^!_{1x} + \Phi^!_{2y} &\text{with probability } \propto M_{\Phi_{1x}}(\dd x)  M_{\Phi_{2y}}(\dd y)\\
        \Phi^!_{1y} + \Phi^!_{2x} &\text{with probability } \propto  M_{\Phi_{1x}}(\dd y)  M_{\Phi_{2y}}(\dd x)\\
        \Phi_1+ \Phi^!_{2(x, y)} &\text{with probability } \propto M_{\Phi_2^{(2)}}(\dd x \, \dd y) \\
    \end{cases}
    \]
\end{corollary}
\begin{proof}
We use the Palm algebra \eqref{eq:palm_algebra} to show that 
\[(\Phi_1 + \Phi_2)^!_{(x, y)}=((\Phi_1 + \Phi_2)^!_x )_y^{!}, \]
and then we apply Theorem \ref{teo:main} twice. As a consequence we obtain the four cases in the statement of the theorem. 
    It remains to assign the probabilities to each of the four cases.
    Consider for instance the first one, by the chain rule its probability equals
    \[
        \frac{M_{\Phi_1}(\dd x)}{M_{\Phi_1}(\dd x) + M_{\Phi_2}(\dd x)} \times \frac{M_{\Phi^!_{1x}}(\dd y)}{M_{\Phi^!_{1x}}(\dd y) + M_{\Phi_2}(\dd y)}
    \]
    which is equal to the one reported in the satament since $M_{\Phi_1^{(2)}}(\dd x \, \dd y) = M_{\Phi^!_{1x}}(\dd y) M_{\Phi_1}(\dd x)$ by Proposition 3.3.9 in \cite{BaBlaKa}. Further, note that by using the same relation, it can be seen that all the denominators of the four different cases are equal.
\end{proof}

\section{Proof of Theorem \ref{teo:main}} \label{sec:proof}

    Let $\Phi = \Phi_1 + \Phi_2$. By Proposition 3.2.1 in \cite{BaBlaKa}, the Palm kernel of a point process $\Phi$ is uniquely characterized by the following relation
    \[
        \partder{t} L_{\Phi}(f + t g)|_{t = 0} = - \int_{\X} g(x) L_{\Phi_x}(f) M_{\Phi}(\dd x)
    \]
    for all measurable $f, g: \X \rightarrow \R_+$ such that $g$ is $M_{\Phi}$-integrable. Here $L_{\Phi}(f)$ denotes the Laplace functional of $\Phi$ evaluated at $f$, i.e., $L_{\Phi}(f) = \E\left[e^{- \int_{\X} f(x) \Phi(\dd x)}\right]$.
    For ease of notation, throughout the proof we write $\Phi(f) := \int_{\X} f(x) \Phi(\dd x)$.

    By Lebesgues' theorem, we have
    \begin{align*}
        \partder{t} L_{\Phi}&(f + t g)|_{t = 0} = \partder{t} \E\left[\e^{-\int_{\X} (f(y) + tg(y)) (\Phi_1 + \Phi_2)(\dd y)} \right]|_{t=0} \\
        &= - \E\left[ \int_{\X} g(x) \e^{-\int_{\X} f(y) (\Phi_1 + \Phi_2)(\dd y)} \Phi_1(\dd x) \right] \\
        & \qquad \qquad - \E\left[ \int_{\X} g(x) \e^{-\int_{\X} f(y) (\Phi_1 + \Phi_2)(\dd y)} \Phi_2(\dd x) \right] \\
        & = -\E\left[\Phi_1(g) \e^{-\Phi_1(f)}\right]\E[\e^{-\Phi_2(f)}] -\E\left[\Phi_2(g) \e^{-\Phi_2(f)}\right]\E[\e^{-\Phi_1(f)}]
    \end{align*}
    where the last equality follows from the independence of $\Phi_1$ and $\Phi_2$.
    Then, applying the Campbell-Little-Mecke formula, we obtain
    \begin{multline*}
        \partder{t} L_{\Phi}(f + t g)|_{t = 0} = \\
        - \left[\int_{\X} g(x) L_{\Phi_{1x}} M_{\Phi_1}(\dd x)\right] L_{\Phi_2}(f) - \left[\int_{\X} g(x) L_{\Phi_{2x}} M_{\Phi_2}(\dd x)\right] L_{\Phi_1}(f)
    \end{multline*}
    and the thesis now follows by multiplying and dividing by $M_{\Phi_1}(\dd x ) + M_{\Phi_2}(\dd x)$.

\bibliography{references}

\end{document}